%% file: ex_article.tex
\documentclass[onefignum,onetabnum]{siamart171218}

\input{ex_shared}

\begin{document}

\maketitle

\begin{abstract}
    We describe a novel algorithm for solving general parametric (nonlinear) eigenvalue problems. Our method has two steps: first, high-accuracy solutions of \emph{non-parametric} versions of the problem are gathered at some values of the parameters; these are then combined to obtain global approximations of the parametric eigenvalues. To gather the non-parametric data, we use non-intrusive contour-integration-based methods, which, however, cannot track eigenvalues that migrate into/out of the contour as the parameter changes. Special strategies are described for performing the combination-over-parameter step despite having only partial information on such \textit{migrating} eigenvalues. Moreover, we dedicate a special focus to the approximation of eigenvalues that undergo bifurcations. Finally, we propose an adaptive strategy that allows one to effectively apply our method even without any \emph{a priori} information on the behavior of the sought-after eigenvalues. Numerical tests are performed, showing that our algorithm can achieve remarkably high approximation accuracy.
\end{abstract}



\section{Introduction}\label{sec:intro}
Consider a generic parametric nonlinear eigenvalue problem
\begin{equation}\label{eq:pevp}
    \vL(\lambda,p)\vx=\vzero,\qquad\vx\neq\vzero.
\end{equation}
Above, $p\in\mathcal{P}$ denotes one or more parameters and
\begin{equation*}
\vL:\bC\times\mathcal{P}\to\bC^{n\times n},\quad n\in\bN,
\end{equation*}
is a matrix-valued function, assumed to be analytic. For all $p\in\mathcal{P}$, we seek \emph{eigenvalues} $\lambda=\lambda(p)\in\bC$ such that an \emph{eigenvector} $\vx=\vx(p)\in\bC^n\setminus\{\vzero\}$, satisfying \cref{eq:pevp}, exists.

This kind of problem differs from ``classical'' nonlinear eigenvalue problems due to the presence of the parameter $p$. In this setting, the eigenvalues become, in fact, eigenvalue \emph{curves}, subject to variations with respect to $p$. As such, solving the parametric problem \cref{eq:pevp} requires identifying a collection of complex-valued (continuous \cite{Lancaster}) functions of $p$, characterizing the evolution of the eigenvalues. See \cref{fig:qual} for an example.

\begin{figure}[t]
\centering
\includegraphics{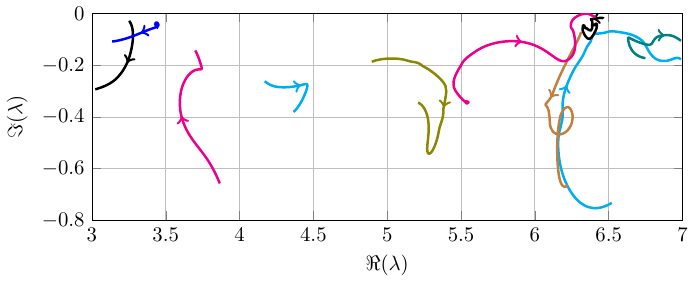}
\caption{Sample view of a collection of eigenvalue curves. The arrows denote the ``direction of movement'' along each curve as $p$ increases. The shown curves have been obtained by applying our proposed approximation algorithm to the numerical example described in \cref{sec:5:3}.}
\label{fig:qual}
\end{figure}

Parametric (nonlinear) eigenvalue problems arise in many engineering and scientific applications. Some examples include electric \cite{NP21}, quantum \cite{proj3}, and delayed-control \cite{Caliskan09,subspace1} systems. In such applications, the parameters are used to model variations (e.g., due to uncertainty) in the underlying physical model, e.g., material properties or boundary conditions.

Although it is not as popular as the non-parametric eigenvalue problem, the numerical solution of parametric eigenvalue problems has been considered before. For instance, perturbation- or continuation-based \cite{Lancaster,cont2,cont3,MoBuOv97,Pl2016} and projection-based \cite{proj3,subspace1,subspace2,subspace3,Xi21} methods have been developed for fairly general problem classes. Still, all these methods are \emph{intrusive}, i.e., they rely on knowledge of, and \emph{access to}, the structure of the function $\vL$ defining \cref{eq:pevp}.

Non-intrusive alternatives, treating the function $\vL$ as a ``black box'', have also been proposed. As notable examples, we mention collocation-based approaches\footnote{For fairness, we note that some of the mentioned references allow arbitrary dependence of $\vL$ on $p$ but assume linearity with respect to $\lambda$. Effectively, they can be considered ``partially intrusive'', covering only the parametric \emph{linear} eigenvalue problem: $\vL(\lambda,p):=\vA(p)-\lambda\vB(p)$.} \cite{NP21,Andreev2012,bertrand_datadriven_2023,hakula_approximate_2015}. 
Non-intrusiveness is an extremely attractive feature, since it makes a method appropriate as an off-the-shelf solver for an extremely wide family of problems. In collocation-based approaches, data is gathered at some values of the parameter (the so-called \emph{collocation points}), in the form of solutions of the eigenvalue problem for fixed $p$. Such data is then used to build approximations of the eigenvalue curves.

Collocation-based methods often struggle with irregular (e.g., bifurcating) eigenvalues, which can naturally arise even for analytic $\vL$. Another issue of these methods is related to the fact that eigenvalues may display large variations with respect to $p$. For instance, the number of computed eigenvalues may differ for two distinct collocation points, as a consequence of what we will call \textit{eigenvalue migrations}. For a more detailed explanation, we refer to \cref{sec:2:2}. This is an intrinsic problem for collocation-based approaches: if the available samples are inconsistent, the corresponding approximation strategy will likely break down or give flawed results.

In this paper, we propose a new non-intrusive collocation-based method, which extends the above-mentioned works in the following directions:
\begin{itemize}
    \item we allow migrating eigenvalues, i.e., eigenvalue curves that may ``appear'' and ``disappear'' as $p$ varies (\cref{sec:2});
    \item we describe a way to recover high approximation quality near bifurcations, despite the lack of regularity of the eigenvalue curves (\cref{sec:3});
    \item we devise strategies for flagging eigenvalue bifurcations (\cref{sec:3:1});
    \item we endow our method with an automatic non-intrusive \emph{adaptive} strategy for selecting the collocation points (\cref{sec:4}).
\end{itemize}
The last item is particularly relevant for practitioners: adaptivity is crucial in problems where no \emph{a priori} information is available on the behavior (e.g., in terms of regularity or migrations) of the sought-after eigenvalue curves. See our numerical examples in \cref{sec:5}. To achieve all this, we develop \emph{ad hoc} procedures, taking full advantage of the structure of the available eigenvalue data on which the method relies.

Throughout this work, we restrict our attention to the single-parameter case: $\mathcal{P}:=[p_{\textrm{min}},p_{\textrm{max}}]\subset\bR$. Moreover, we focus on finding only the eigen\emph{values}, disregarding eigen\emph{vectors}. The motivation behind this choice is that, once the eigenvalues are available, finding the eigenvectors is usually straightforward, requiring only the identification of the null space of a matrix. A discussion on possible extensions to more parameters and to eigenvector approximation is included in \cref{sec:conclusions}.

\section{Solving non-parametric eigenvalue problems by contour integration}\label{sec:1}
Before considering the parametric eigenvalue problem \cref{eq:pevp}, we look at a non-parametric one:
\begin{equation}\label{eq:nep}
    \vF(\lambda)\vx=\vzero,\qquad\vx\neq\vzero,
\end{equation}
with $\vF:\bC\to\bC^{n\times n}$ analytic. As outlined above, and as will become more apparent in \cref{sec:2}, our proposed method for addressing \emph{parametric} eigenvalue problems relies on collecting solutions of \emph{non-parametric} eigenvalue problems, whose matrices $\vF(\cdot)$ coincide with $\vL(\cdot,p_j)$, for $p_1,\ldots,p_S\in\mathcal{P}$. For this reason, solving problems like \cref{eq:nep} accurately is crucial for our ultimate parametric endeavor.

Due to the nonlinear nature of $\vF$, this problem can have finitely many, infinitely many, or no solutions. For a thorough overview of this kind of problem, we refer to \cite{GuTi17}, where several effective algorithms for solving \cref{eq:nep} are also presented. Among them, here we restrict our attention to contour-integration-based eigenvalue solvers. This choice is due to the versatility of these methods: in line with our (parametric) objectives, no \emph{a priori} knowledge of the structure of $\vF$ is needed to apply them.

Many different flavors of contour-integration-based eigenvalue solvers are available in the literature. In particular, we recall the \emph{FEAST} algorithm \cite{GaMiPo18,Po09}, the Sakurai-Sugiura method \cite{Asetal10,SaSu03}, Beyn's method \cite{Be12}, and Loewner-framework-based solvers \cite{Bretal23}. All the above-mentioned methods compute eigenvalues of $\vF$ in a specific user-defined region of the complex plane, namely, within an open set $\Omega\subset\bC$. 
In order to do this, a fundamental tool is the following theorem, which is a simplified version of, e.g., \cite[Theorem 2.8]{GuTi17}.
\begin{theorem}[Keldysh]\label{th:keldysh}
    Let $\vF\colon\bC\rightarrow\bC^{n\times n}$ be analytic over the domain $\Omega\subset\bC$, and assume that its eigenvalues in $\Omega$ are $\{\lambda^1,\dots,\lambda^N\}$. Repeated eigenvalues are allowed, to account for multiplicity. Then there exist matrices $\vX,\vY\in\bC^{n\times N}$ and $\vJ\in\bC^{N\times N}$, as well as a  matrix-valued function $\vH:\bC\to\bC^{n\times n}$, analytic over $\Omega$, such that
    \begin{equation*}
        \vF(z)^{-1}=\underbrace{\vX(z\vI-\vJ)^{-1}\vY^{\mathrm{H}}}_{\vG(z)}+\vH(z).
    \end{equation*}
    Specifically,
    \begin{equation*}
    \vJ=\begin{bmatrix}
        \lambda^1 & \star & & \\[-1mm]
        & \lambda^2 & \ddots & \\[-1mm]
        & & \ddots & \hspace{-1mm}\star \\[1mm]
        & & & \lambda^N
    \end{bmatrix},
    \end{equation*}
    with each of the super-diagonal entries (denoted by a star) being either $0$ or $1$. Any $1$-entry relates to a defective eigenvalue.
\end{theorem}
\cref{th:keldysh} states that the inverse of $\vF$ can be decomposed (over $\Omega$) as the sum of two functions, one rational and one analytic. The eigenvalues of interest are poles of the rational term $\vG$. As such, the term $\vH$ has to be filtered out. To this aim, we consider a contour integral on $\Gamma:=\partial\Omega$, involving an analytic so-called \textit{filter function} $f:\Omega\to\bC$, and then employ Cauchy's integral formula \cite[Definition 1.11]{Hi08} to obtain
\begin{align*}
\frac1{2\pi\mathrm{i}}\oint_\Gamma f(z)\vF(z)^{-1}\mathrm{d}z=&\frac1{2\pi\mathrm{i}}\oint_\Gamma f(z)\vG(z)\mathrm{d}z+\frac1{2\pi\mathrm{i}}\oint_\Gamma f(z)\vH(z)\mathrm{d}z\\
=&\frac1{2\pi\mathrm{i}}\oint_\Gamma f(z)\vX(z\vI-\vJ)^{-1}\vY^{\mathrm{H}}\mathrm{d}z=\vX f(\vJ)\vY^{\mathrm{H}}.
\end{align*}
The above formula allows one to identify eigenvalues of $\vF$ within $\Omega$ from the contour integral on the left-hand side. As such, it is the starting point of all contour-integration-based eigenvalue solvers.

For instance, in Beyn's method \cite{Be12}, given two integers $K,m\in\bN$ and a random $(n\times m)$-matrix $\vR$, one computes $2K$ matrices of size $n\times m$:
\begin{equation}\label{eq:beyn}
    \vA_k=\frac1{2\pi\mathrm{i}}\oint_\Gamma z^k\vF(z)^{-1}\vR\mathrm{d}z,\quad\text{for }k=0,1,\ldots,2K-1.
\end{equation}
Then two block-Hankel matrices of size $(Kn)\times(Km)$ are formed:
\begin{equation*}
    \vB_k=\begin{bmatrix}
        \vA_k & \vA_{k+1} & \cdots & \vA_{k+K-1}\\
        \vA_{k+1} & \vA_{k+2} & \cdots & \vA_{k+K}\\
        \vdots & \vdots & & \vdots\\
        \vA_{k+K-1} & \vA_{k+K} & \cdots & \vA_{k+2K-2}
    \end{bmatrix},\quad\text{for }k=0,1,
\end{equation*}
and a thin SVD of $\vB_0=\vU\vSigma\vV^{\mathrm{H}}$ is computed. Under minor conditions on the ranks of $\vB_0$ and $\vB_1$, simple algebraic manipulations \cite[Theorem 5.2]{GuTi17} show that the matrix $\vJ$ (with the sought-after eigenvalues on its diagonal) appears in the Jordan normal form of the easily computable matrix $\vU^{\mathrm{H}}\vB_1\vV\vSigma^{-1}$.

It is important to note that, in practice, the integrals in \cref{eq:beyn} must be computed approximately, through a quadrature rule. For instance, if $\Gamma$ is the radius-$r$ circle centered at $z_0$, we can use the trapezoidal rule to approximate
\begin{equation}\label{eq:beyn_quad}
\frac1{2\pi\mathrm{i}}\oint_\Gamma f(z)\vF(z)^{-1}\vR\mathrm{d}z\approx\frac1{N_z}\sum_{j=1}^{N_z}(z_j-z_0)f(z_j)\mathbf{F}(z_j)^{-1}\vR,
\end{equation}
with $z_j=z_0+r\exp\{2\pi\textrm{i}j/N_z\}$, $j=1,\ldots,N_z$, being uniformly spaced quadrature nodes on $\Gamma$. We refer to \cite{GuTi17,Be12} for more details on the implementation, stability, accuracy, and complexity of Beyn's method.

Before proceeding further, it is important to note the following.

\begin{remark}\label{rem:localanalytic}
    In \cref{th:keldysh}, regularity of $\vF$ is required only over $\Omega$, and not over the whole $\bC$. Effectively, due to its improved generality, this is our working assumption both for $\vF$ in the non-parametric problem \cref{eq:nep} and for $\vL$ in the parametric problem \cref{eq:pevp}. More precisely, $\vL$ is assumed analytic over $\Omega\times\mathcal{P}$.
\end{remark}

\section{Match-based parametric eigenvalue solver}\label{sec:2}

In this section, we describe our proposed approach bottom-up. Throughout this portion of the paper, we make the following assumptions: (i) all eigenvalues stay semi-simple for all values of $p$ and (ii) the collocation points are fixed in advance. These assumptions will be later removed in \cref{sec:3,sec:4}, respectively, showing the full potential of our method.

\subsection{Basic match strategy}\label{sec:2:1}

Here we look at the simplest possible instance of our method. Specifically, we consider the parametric eigenvalue problem \cref{eq:pevp} and choose \emph{two} different values of $p$, namely, $p_1$ and $p_2$. By fixing $p=p_j$, $j=1,2$, \cref{eq:pevp} becomes a \emph{non-parametric} nonlinear eigenvalue problem, solvable by a plethora of available techniques, cf.~\cref{sec:1}. Applying any such method at $p=p_1$ and $p=p_2$ yields two collections of (approximate) eigenvalues: $\{\lambda_j^i\}_{i=1}^{N_j}$, $j=1,2$. The subscript $j$ and the superscript $i$ index the different sampled $p$-points (to which we simply refer as \textit{collocation points}) and the different eigenvalues, respectively.

Before proceeding further, it is crucial to make the following assumption, through which, essentially, we require the input data of our algorithm to be reliable.

\begin{hypothesis}\label{as:exactness}
The non-parametric eigenvalue solver of choice is sufficiently accurate. Specifically, we assume that:
\begin{itemize}
    \item Non-parametric eigenvalues are identified with very low error levels, so that we may use them as ``ground truth''.
    \item Spurious effects are absent. For contour-integration-based methods (on which we focus here), this means that all eigenvalues within the region $\Omega$ are approximated with their correct (algebraic) multiplicity.
\end{itemize}
\end{hypothesis}

Note that, in general, the number of eigenvalues may be different: $N_1\neq N_2$. An in-depth discussion on this is postponed till \cref{sec:2:2}. For now, we assume that $N:=N_1=N_2$.

The ultimate goal of our algorithm is to build a family of $N$ curves $\{\tilde\lambda^i\}_{i=1}^N$: for each $i$, our approximation of the $i$-th eigenvalue curve is a complex-valued function
\begin{equation*}
\tilde\lambda^i:\mathcal{P}\to\bC.
\end{equation*}
Specifically, for every $p\in\mathcal{P}$, we wish for $\tilde\lambda^i(p)$ to be an approximate eigenvalue of \cref{eq:pevp}, for all $i=1,\ldots,N$. In the present case, since data is given only at two collocation points, our approximate eigenvalue curves are straight lines: for all $i$, $\tilde\lambda^i(p):=a^i+b^ip$, for some coefficients $a^i$ and $b^i$ to be found.

In our algorithm, we find such coefficients by imposing interpolation of the data: for every $i=1,\ldots,N$ and $j=1,2$, there must exist an index $k=k(i,j)$ such that $\tilde\lambda^k(p_j)=\lambda_j^i$. Note that we \emph{do not} require that $k=i$, since the two lists of eigenvalues may be indexed in an inconsistent way. For instance, the eigenvalue $\lambda_1^4$ might need to be paired up with $\lambda_2^3$ and, to this aim, we use a curve, say, $\tilde\lambda^6$, which satisfies the interpolation conditions $\tilde\lambda^6(p_1)=\lambda_1^4$ and $\tilde\lambda^6(p_2)=\lambda_2^3$.

For the reason just described, a preliminary step to the construction of the approximate eigenvalues $\tilde\lambda^i$ is re-indexing the data sets to make them coherent. 
Since we carry out this step by comparing the eigenvalue data 1-to-1, we refer to it as \emph{matching} the two sets of eigenvalues. More precisely, we build an $(N\times N)$-matrix $C$, where each entry $c_{ik}$ represents the distance between the $i$-th eigenvalue at $p=p_1$ and the $k$-th eigenvalue at $p=p_2$:
\begin{equation}\label{eq:cost_entry}
    c_{ik}:=\abs{\lambda_1^i-\lambda_2^k}.
\end{equation}
Then we select $N$ entries of $C$, one per row and one per column, so that their sum is minimal among all such choices. More formally, we seek two permutations of the set $\{1,\ldots,N\}$, say, $\bm\sigma$ and $\bm\tau$, that achieve the minimum of the \emph{loss function} 
\begin{equation}\label{eq:lossfunction}
\mathcal{L}({\bm\sigma},{\bm\tau})=\sum_{i=1}^Nc_{\sigma_i\tau_i}=\sum_{i=1}^N\abs{\lambda_1^{\sigma_i}-\lambda_2^{\tau_i}}
\end{equation}
over all such permutations. This can be done very efficiently, e.g., by the so-called \textit{Hungarian algorithm} \cite{Hungarian} or by casting the problem as a minimum-cost flow problem on a bipartite graph \cite{Crouse2016}.

The indices of such optimal entries yield the desired match. Specifically, for all $i=1,\ldots,N$, the $\sigma_i$-th eigenvalue at $p=p_1$ is matched to the $\tau_i$-th eigenvalue at $p=p_2$. Now all the ingredients are ready for building the approximate eigenvalue curves: we define one curve $\tilde\lambda^i$ for each $(\sigma_i,\tau_i)$ pair, as the only straight line that passes through $(p_1,\lambda_1^{\sigma_i})$ and $(p_2,\lambda_2^{\tau_i})$. A graphical representation of this is shown in \cref{fig:match}.

\begin{remark}\label{rem:vectors}
    The loss function in \eqref{eq:lossfunction} can be generalized to incorporate the eigenspace dynamics as an additional error term: if $\vx_j^i$ is a normalized eigenvector corresponding to $\lambda_j^i$ at $p=p_j$, then we may consider the loss function
    \begin{equation*}
    \mathcal{L}({\bm\sigma},{\bm\tau})=\sum_{i=1}^N\Big(\abs{\lambda_1^{\sigma_i}-\lambda_2^{\tau_i}}-\gamma\abs{\vx_1^{\sigma_i}\cdot\vx_2^{\tau_i}}\Big),
    \end{equation*}
    with $\gamma$ being a positive weight and $\cdot$ denoting the dot product (obviously, other notions of closeness between eigenspaces may be used instead). The additional eigenspace-related information can be extremely helpful in the matching process, e.g., in the case of eigenvalue crossings. However, including this additional term can sometimes add difficulties, e.g., related to a proper choice of the weight $\gamma$ or to comparing eigenvectors of different sizes (cf.~our numerical test in \cref{sec:5:3}). For the sake of simplicity (and also because it is not the main focus of this work), we ignore the eigenspace term in our discussion.
\end{remark}

\begin{figure}[t!]
\centering
\includegraphics{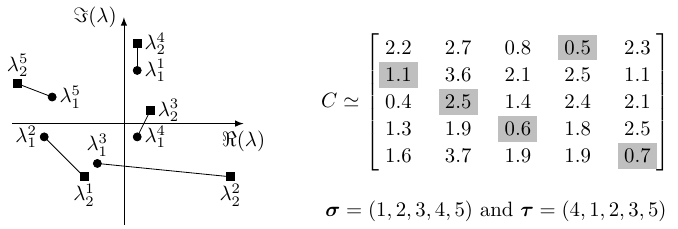}
\caption{Example of match and approximation of the eigenvalue curves. Exact eigenvalues at $p=p_1$ and $p=p_2$ are shown as circles and squares, respectively. The line segments are the approximate eigenvalue curves $\tilde\lambda^i$. On the right, the cost matrix (with entries truncated at the first decimal digit). The highlighted entries correspond to the optimal match, which is reported below the matrix.}
\label{fig:match}
\end{figure}

\subsection{Dealing with migrating eigenvalues}\label{sec:2:2}

The eigenvalues at the two collocation points $p_j$, $j=1,2$, (the data on which our method relies) may generally consist of different numbers $N_j$ of eigenvalues. For instance, if $\Delta N:=N_1-N_2>0$, more eigenvalues are available at $p=p_1$ than at $p=p_2$. For the sake of our discussion (and in line with \cref{as:exactness}), it is crucial to assume that these $\Delta N$ extra eigenvalues are not \emph{spurious}, but, in fact, they correspond to eigenvalues that leave the integration region $\Omega$ as $p$ varies from $p_1$ to $p_2$. We refer to this phenomenon as \textit{eigenvalue migration}.

From a practical viewpoint, a nonzero $\Delta N$ is a consequence of the adaptive behavior of the non-parametric eigenvalue solver of choice. For instance, based on \cref{th:keldysh}, contour-integration-based eigenvalue solvers are designed to approximate only eigenvalues \emph{within} the integration region $\Omega$. Effectively, this means that any eigenvalues that migrate out of the contour $\Gamma$ at some $p\in(p_1,p_2)$, see \cref{fig:migration} (left), are invisible at $p=p_2$.

Eigenvalue migrations do not affect only contour-integration-based approaches: several solvers for nonlinear eigenvalue problems adaptively select the number of eigenvalues through the use of, e.g., rank-based indicators \cite{proj3,Bretal23,subspace5} or other estimators \cite{pradovera_adaptive_2022}. Obviously, this can lead to eigenvalue migrations.

At this point, it is important to note that $\Delta N$ is the \emph{net} number of migrations, i.e., the difference between the outgoing and incoming eigenvalues. E.g., if $\Delta N\geq0$, given information only at $p_1$ and $p_2$, it is impossible to tell whether we have exactly $\Delta N$ outgoing eigenvalues, or $(\Delta N+1)$ outgoing and $1$ incoming eigenvalues, etc. As such, it is imperative to assume, at least for now, that, if $\Delta N\geq0$ (resp., $\Delta N\leq0$) there are no incoming (resp., outgoing) eigenvalues. We will later discuss how to solve this ambiguity in \cref{sec:2:3}, by leveraging data at further values of $p$ between $p_1$ and $p_2$.

If $\Delta N\neq 0$, the optimization problem that defines the permutations $\bm\sigma$ and $\bm\tau$ becomes asymmetric: $\bm\sigma$ and $\bm\tau$ are now permutations of the sets $\{1,\ldots,N_1\}$ and $\{1,\ldots,N_2\}$, respectively, attaining the minimum of
\begin{equation*}
\mathcal{L}({\bm\sigma},{\bm\tau})=\sum_{i=1}^{\min\{N_1,N_2\}}c_{\sigma_i\tau_i}=\sum_{i=1}^{\min\{N_1,N_2\}}\abs{\lambda_1^{\sigma_i}-\lambda_2^{\tau_i}}.
\end{equation*}
This version of the problem also admits an efficient solution, e.g., by the Hungarian algorithm.

Now, without loss of generality\footnote{If $\Delta N<0$, our discussion still applies, provided the roles of $p_1$ and $p_2$ are swapped.}, let us assume $\Delta N>0$, i.e., we consider $N_1>N_2$. The first $N_2$ entries of $\bm\sigma$ and $\bm\tau$ form a balanced match. As such, we can deal with them as described in the previous section, resulting in $N_2$ approximate eigenvalue curves $\{\tilde\lambda^i(p)\}_{i=1}^{N_2}$. On the other hand, the last $\Delta N$ entries of $\bm\sigma$ are left unused: the eigenvalues with indices $\sigma_{N_2+1},\ldots,\sigma_{N_1}$ at $p=p_1$ are not matched to any eigenvalue at $p=p_2$, since they do not appear in the loss function. For this reason, the remaining $\Delta N$ eigenvalue curves are missing the final eigenvalue in their trajectory, complicating the approximation effort.

Our solution is the following: for each leftover eigenvalue at $p=p_1$, we create a fictitious eigenvalue at $p=p_2$, which we use to define our linear approximation as usual. To model eigenvalue migration, such made-up eigenvalues should be located outside the integration region, so that the approximate eigenvalue curves also migrate as $p\to p_2$. In general, a different fictitious eigenvalue may be defined for each of the remaining $\Delta N$ eigenvalues. However, for simplicity, we always pick the same one, namely, $\overline\lambda=\infty$.

\begin{figure}[t]
\centering
\includegraphics{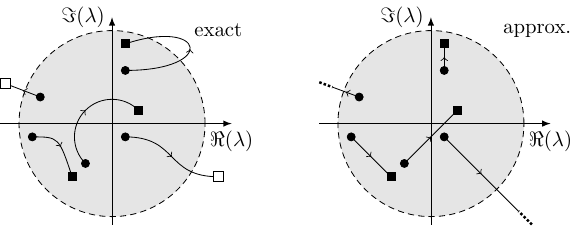}
\caption{Example of migrating eigenvalues. Arrows show the direction of eigenvalue changes as $p$ increases. Eigenvalues at $p=p_1$ and $p=p_2$ are represented as circles and squares, respectively. The empty squares are not captured by the non-parametric eigenvalue solver at $p=p_2$. Left plot: the curves denote the continuous evolution of the eigenvalues as $p$ varies from $p_1$ to $p_2$. Right plot: the line segments correspond to the approximate eigenvalues $\tilde\lambda^i$, assuming a correct match. Note that our algorithm is unaware of the curves: data is provided only at $p=p_j$, $j=1,2$. In the approximation, two eigenvalues are mapped to $\infty$ along radial lines.}
\label{fig:migration}
\end{figure}

Note that this choice is only formal, since the corresponding approximate eigenvalue curve, being the linear interpolation of a finite and an infinite value, always yields $\infty$. So, for all migrating eigenvalues, we replace the meaningless linear interpolation
\begin{equation*}
    \tilde\lambda^i(p)=\frac{p_2-p}{p_2-p_1}\lambda_1^{\sigma_i}+\frac{p-p_1}{p_2-p_1}\underbrace{\overline\lambda}_\infty=\infty
\end{equation*}
with a projective-geometry-inspired weighted harmonic mean:
\begin{equation*}
    \tilde\lambda^i(p)=\bigg(\frac{p_2-p}{p_2-p_1}(\lambda_1^{\sigma_i})^{-1}+\frac{p-p_1}{p_2-p_1}\underbrace{(\overline\lambda)^{-1}}_0\bigg)^{-1}=\frac{p_2-p_1}{p_2-p}\lambda_1^{\sigma_i}.
\end{equation*}
If the contour is a circle centered at the origin, this choice corresponds to approximated eigenvalues that move along radial lines from the origin\footnote{To recover this radial behavior when $\Omega$ is centered at some other $z_0\neq 0$, we suggest using a shifted version of the weighted harmonic mean: $\tilde\lambda^i(p)=\frac{p_2-p_1}{p_2-p}\left(\lambda_1^{\sigma_i}-z_0\right)+z_0$.}. An example of the approximate eigenvalues obtained with this method is shown in \cref{fig:migration} (right).

%Before proceeding, we make the following remark.

\begin{remark}\label{rem:cutoff}
    Any approximated eigenvalue outside the contour should not be trusted, since it might be affected by large errors. Specifically, any such approximate eigenvalue should be dropped. In practice, whenever $\Delta N\neq0$ as above, unbalanced eigenvalues never actually reach $\overline\lambda$. Rather, they disappear as soon as they cross the contour $\Gamma$.
\end{remark}

\subsection{From two to many}\label{sec:2:3}

Above, we have described how approximated eigenvalues can be obtained from two collocation points. If more collocation points are available, our construction can be easily generalized. To this aim, assume that a (not necessarily uniform) sampling grid $p_1<\ldots<p_S$ is available. Let $N_j$ be the number of eigenvalues identified at $p=p_j$ by the non-parametric eigenvalue solver of choice. Our ultimate target is the construction of a set of $N:=\max_{j=1,\ldots,S}N_j$ curves $\{\tilde\lambda^i\}_{i=1}^N$, each globally defined over $\mathcal{P}$, interpolating the available data.

To this aim, we first carry out the matching step from \cref{sec:2:1} over each interval of the form $[p_j,p_{j+1}]$, for $j=1,\ldots,S-1$, allowing us to reorder the local information in a coherent way. In case of eigenvalue migrations, we adopt the strategy proposed in \cref{sec:2:2}, placing fictitious eigenvalues at $\infty$. Then, the global approximate eigenvalue curves are obtained by interpolating each eigenvalue over the whole $p$-range, following the precomputed (locally optimal) matches. Note that approximate eigenvalues may migrate away to $\infty$ at some intermediate value of $p$, only to later return inside the contour.

The interpolation with respect to $p$ can be carried out with arbitrary strategies over the grid of collocation points, e.g., using piecewise-linear hat functions, splines, or radial basis functions, just to mention a few possibilities. In particular, note that interpolation schemes with a larger \emph{stencil} (e.g., high-order splines or radial basis functions with large support) have the potential of achieving higher accuracy in regions where the eigenvalue curves are smooth. This is notably the case if migrations and bifurcations are absent.

Regardless of the interpolation strategy of choice, one should be careful when dealing with fictitious eigenvalues at $\infty$. One option of handling this was described in \cref{sec:2:2}, and can be applied, e.g., in the piecewise-linear case, replacing straight-line segments with hyperbolic segments.

However, whenever more than 2 collocation points are present, an alternative option can be considered for dealing with migrating eigenvalues: one can build an approximate eigenvalue curve relying on data \emph{before} the migration, and then \emph{extrapolate} it to predict the eigenvalue location as the migration happens. In practice, this simply corresponds to ignoring any fictitious ``$\infty$'' eigenvalues, thus building the curves $\tilde\lambda^i$ by interpolating \emph{finite} eigenvalue data only. Such curves are then extrapolated even at values of $p$ for which migrations happen, leading to global approximations. Of course, in this step, any approximate eigenvalue that is predicted to be outside $\Omega$ should be dropped, as described in \cref{rem:cutoff}. Overall, this extrapolation-based strategy has the potential of being much more accurate than the crude ``harmonic mean'' approach described in \cref{sec:2:2}. As such, we choose it as our \emph{de facto} standard in our numerical implementation.

\begin{remark}\label{rem:migrations}
    Practically speaking, any eigenvalue at $\infty$ represents a ``barrier'' in the interpolation step, which generally degrades the approximation quality near the migration point (regardless of which approach for predicting migrating eigenvalues is chosen). For this reason, interpolation schemes with small stencils (e.g., piecewise-linear interpolation or splines of low order) may be preferred if many migrations are present, usually at the price of a decreased accuracy \emph{also away from migrations}.
    
    This being said, in our numerical tests in \cref{sec:5}, the extrapolation-based approximation strategy described above turns out to be extremely effective at preserving the approximation quality, despite any migration-induced discontinuities in the eigenvalue curves. This empirical evidence seems to suggest that one should \emph{not} avoid higher-order interpolation schemes (over larger interpolation stencils) for fear of migrations, as long as enough data is available to guarantee a stable extrapolation.
\end{remark}

\begin{figure}[t]
\centering
\includegraphics{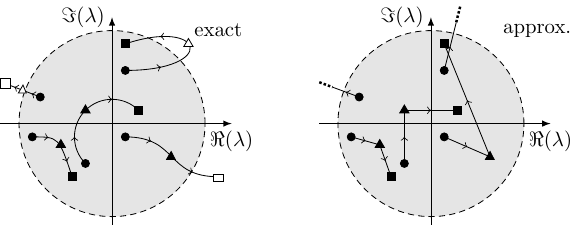}
\caption{Example of piecewise-linear global approximation from $3$ collocation points $p_1<p_2<p_3$. Eigenvalues at $p=p_j$, $j=1,2,3$, are represented as circles, triangles, and squares, respectively. The empty shapes are not captured by the non-parametric eigenvalue solver. Left plot: the curves denote the continuous evolution of the eigenvalues as $p$ varies from $p_1$ to $p_3$. Right plot: approximate eigenvalues $\tilde\lambda^i$. A single incorrect match is performed on the right-hand side of the plot due to simultaneous exit/entrance migration events.}
\label{fig:migration_triple}
\end{figure}

We show in \cref{fig:migration_triple} an example of the described method, relying on piecewise-linear interpolation. Several eigenvalue migrations can be observed, since $N_1=5>N_2=N_3=3$. Accordingly, in the approximation shown in \cref{fig:migration_triple} (right), we can see that $2=N_1-N_2$ eigenvalues leave the disk\footnote{Note that, in this example, it is impossible to predict migrations through extrapolation, since too little ``finite'' data is available for each of the two migrating eigenvalues. As such, we fall back to the ``harmonic mean'' strategy from \cref{sec:2:2}.} as $p$ increases from $p_1$ to $p_2$. 

In fact, two more migrations are present, as shown in \cref{fig:migration_triple} (left): as $p$ increases from $p_2$ to $p_3$, one eigenvalue in the fourth quadrant leaves the disk, while another eigenvalue re-enters from the first quadrant. These two events are not captured by the approximated curves, since they happen over the same local interval $[p_2,p_3]$, thus canceling each other out. Instead, the outgoing eigenvalue is matched to (and interpolated with) the incoming one, resulting in an erroneous eigenvalue curve.

This is a fundamental issue of our method for dealing with migrations, since an eigenvalue can be flagged as migrating (hence, linked to $\infty$) only if it cannot be matched to any other (finite) eigenvalue. This flaw is intrinsically due to the poor quality of the data. Specifically, in the case displayed in \cref{fig:migration_triple}, the samples are too few for the resulting approximation to have good quality. As such, our proposed strategy for dealing with this issue is the adaptive sampling described in \cref{sec:4}.

\section{Dealing with bifurcations}\label{sec:3}
A naturally arising phenomenon in parametric eigenvalue problems is that of \textit{bifurcations}, namely, values of $p$ where 2 or more eigenvalue curves coalesce and the eigenvalues become defective \cite{Lancaster,Rellich,Kato}. Note that bifurcations are distinct from \textit{(regular) eigenvalue crossings}, i.e., values of $p$ where 2 or more eigenvalue curves cross while the eigenvalues remain (semi-)simple.

For instance, the $(2\times 2)$-valued function
\begin{equation*}
    \vL(\lambda,p)=\begin{bmatrix}\lambda & \\ & \lambda-p\end{bmatrix}
\end{equation*}
has spectrum $\Lambda(p)=\{0,p\}$, and its two eigenvalues have a (regular) crossing at $p=0$. Notably, the eigenvalues \emph{can} be labeled in such a way that the curves \emph{remain smooth} through the crossing point: $\lambda^1(p)=0$ and $\lambda^2(p)=p$.

On the other hand, as a simple example of bifurcation, consider the $(2\times 2)$-valued function
\begin{equation}\label{eq:smallpnep}
    \vL(\lambda,p)=\begin{bmatrix}\lambda & p \\ 1 & \lambda\end{bmatrix},
\end{equation}
whose spectrum is $\Lambda(p)=\{\pm\sqrt{p}\}$. The two eigenvalues are real for $p>0$ and imaginary for $p<0$, with an order-2 bifurcation at $p=0$. Specifically, at $p=0$ the only eigenvalue $\lambda^1(0)=\lambda^2(0)=0$ is defective, and the eigenvalue curves $\lambda^{1,2}(p)=\pm\sqrt{p}$ are continuous but not continuously differentiable there, despite $\vL$ being analytic. In particular, there is no way to label the eigenvalues so that the eigenvalue curves are smooth at $p=0$.

Our method from \cref{sec:2} would greatly struggle with a problem like \eqref{eq:smallpnep}. Essentially, whenever defects are present, treating the eigenvalues as semi-simple, hence trying to approximate irregular curves with regular ones, can lead to gross approximation errors. Also, in the scope of adaptive sampling (as in \cref{sec:4}), this can significantly slow down convergence. Similar issues were already described in \cite[Section 4.2]{NP21}.

We show a specific example in \cref{fig:bif_qual} (left), where a sample order-3 bifurcation is displayed. Although it might seem like three \emph{smooth} eigenvalue curves are crossing at the bifurcation point, each eigenvalue curve has a \emph{non-smooth} $\sqrt[3]{p}$-like cusp. Our ``vanilla'' method yields the approximation showcased in \cref{fig:bif_qual} (center), which completely misses the bifurcation point. On the other hand, the alternative approach that we are about to describe gives much better qualitative results, see \cref{fig:bif_qual} (right). (We mention that, for fairness, we do not include a similar figure for the example \cref{eq:smallpnep}, since, in that case, our upcoming approach yields the \emph{exact} solution.)

\begin{figure}[t]
\centering
\includegraphics{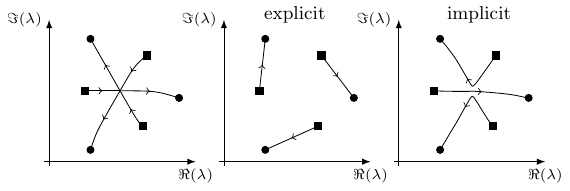}
\caption{Example of local behavior around an order-3 bifurcation (left plot). Arrows show the direction of eigenvalue changes as $p$ increases. Squares and dots denote eigenvalues at $p=p_1$ and $p=p_2$, respectively. Performing a match and linking the eigenvalues 1-to-1 (as described in \cref{sec:2:1}) leads to the middle figure. The all-at-once implicit approximate representation of the eigenvalues (described in \cref{sec:3:0}) gives the right figure.}
\label{fig:bif_qual}
\end{figure}

\subsection{Implicit representation of bifurcating eigenvalues}\label{sec:3:0}
We now describe our proposed approach for dealing with bifurcations. The main underlying idea is the removal of the assumption that all eigenvalues are semi-simple. Specifically, since the \emph{explicit} 1-by-1 approximation of the eigenvalue curves fails due to their lack of smoothness, we rely on an \emph{implicit but smooth} representation, using suitable polynomials as proxies.

Assume that a bifurcation happens at some $p\in(p_1,p_2)$, involving the $M$ eigenvalues $\{\lambda_1^i\}_{i=1}^M$ at $p=p_1$, and $\{\lambda_2^i\}_{i=1}^M$ at $p=p_2$. (Other semi-simple eigenvalues $\lambda^{M+1},\lambda^{M+2},\ldots$ can be dealt with as in \cref{sec:2}.)

As a first step, we define two degree-$M$ monic polynomials, whose roots are the above eigenvalues: $\xi_j(\lambda):=\prod_{i=1}^M(\lambda-\lambda_j^i)$ for $j=1,2$. These polynomials are the above-mentioned proxies that we use to manipulate the $M$ bifurcating eigenvalues all at once. Then, a bivariate polynomial $\xi$ is defined by interpolating $\xi_1$ and $\xi_2$ with respect to $p$: for instance, with linear interpolation,
\begin{equation}\label{eq:ppoly}
    \xi(\lambda,p):=\sum_{j=1}^2\frac{p-p_{3-j}}{p_j-p_{3-j}}\xi_j(\lambda).
\end{equation}
This allows us to have an \emph{implicit parametric} approximation of the eigenvalues: since we are using polynomials as proxies for their roots, we define $\tilde\lambda^1(p),\ldots,\tilde\lambda^M(p)\in\bC$ as the roots of $\xi(\cdot,p)$ (counted with multiplicity) in arbitrary order. In practical terms, this means that, if one seeks a prediction of the eigenvalues at some given $p$, one must first assemble the polynomial $\lambda\mapsto\xi(\lambda,p)$ according to \cref{eq:ppoly}, and then solve a degree-$M$ root-finding problem.

\begin{remark}\label{rem:rootfinding}
    An implicit definition of the eigenvalues is less convenient than the explicit one from \cref{sec:2}. Still, the extra root-finding step is cheap, since it involves only a degree-$M$ scalar polynomial.
\end{remark}

Since the roots of parametric polynomials can bifurcate (with order up to the degree of the polynomial), with the above approach we allow the approximate curves to have bifurcations, thus mimicking the structure of the exact eigenvalue curves. In this way, we can hope to significantly improve the approximation accuracy.

\begin{remark}\label{rem:fullimplicit}
    The approximate eigenvalue curves can be defined implicitly, regardless of whether bifurcations are present. For instance, given $S$ collocation points, consider the data $\{\lambda_j^i\}_{i=1,j=1}^{N,S}$, cf.~\cref{sec:2:1}, as well as a Lagrangian basis $\{\ell_j\}_{j=1}^S$, satisfying $\ell_j(p_{j'})=\delta_{jj'}$, the Kronecker delta. We can define the $N$ approximate curves $\tilde\lambda^1,\ldots,\tilde\lambda^N$ implicitly, as the roots (with respect to $\lambda$, for fixed $p$) of the bivariate polynomial
    \begin{equation*}
        \xi(\lambda,p):=\sum_{j=1}^S\ell_j(p)\prod_{i=1}^N(\lambda-\lambda_j^i).
    \end{equation*}
    Still, our match-based approach from \cref{sec:2} can be considered superior to this alternative option from several viewpoints:
    \begin{itemize}
        \item its approximations of the eigenvalues are explicit and faster to obtain, not requiring any additional root-finding;
        \item it is potentially more stable, avoiding the need to compute roots of (high-degree) polynomials, which, in fact, could even display large non-physical variations with respect to $p$;
        \item it can handle eigenvalue migrations, i.e., changes of degree of the univariate polynomials $\xi_j$; this is the main focus of this work, making this point particularly relevant.
    \end{itemize}
    As such, we choose to rely on an implicit representation of the approximate eigenvalue curves only locally near bifurcations.
\end{remark}

Lastly, we wish to mention that a hybrid explicit-implicit modeling of the eigenvalue curves gives rise to additional difficulties when more than two collocation points are involved. For instance, consider problem \cref{eq:smallpnep}, and assume that a uniform grid of 9 collocation points $p_1=-1$, $p_2=1$, $\ldots$, $p_9=15$, is given. A bifurcation is present between $p_1$ and $p_2$, and the two eigenvalue curves $\tilde\lambda^i$, $i=1,2$ are modeled implicitly for $p\in[p_1,p_2]$. However, $\tilde\lambda^1$ and $\tilde\lambda^2$ admit explicit formulas for $p\in[p_2,p_9]$, since they can be treated as semi-simple there. Then, should the two curves be continued implicitly or explicitly (in a piecewise fashion) for $p>p_2$?

The former option is potentially more accurate, but breaks down if some $\tilde\lambda^i$, $i=1,2$, is later involved in further bifurcations (involving other eigenvalues). Instead, the latter option is more numerically stable, at the price of accuracy, since implicit (for $p<p_2$) and explicit (for $p>p_2$) curves must be ``joined'' at $p_2$. In our implementation, see \cref{sec:5}, we choose an intermediate option: we extend the implicit definition of the curves for just a few more points, until, say, $p_5$. This is done in order to guarantee a good approximation quality near the bifurcation, employing a wider stencil of half-width\footnote{The desired stencil width can be considered an additional user input.} 4, as opposed to just 1. Then, the standard explicit representation is used outside this stencil, namely, for $p>p_5$.

\subsection{Flagging bifurcations}\label{sec:3:1}
Before we can deploy the hybrid explicit-implicit construction described above, one major issue remains: how can bifurcations (with the corresponding bifurcation orders) be identified? Specifically, in our framework, such bifurcation-flagging procedure can rely only on information at the collocation points. We proceed by describing one such strategy.

Assume that eigenpair data is available at two collocation points $p_1$ and $p_2$. We wish to determine whether any bifurcation happens at some $p\in[p_1,p_2]$ and, if so, which eigenpairs are involved. For this, we apply the following criterion, which, for now, we state only informally and without proof: we suspect that bifurcations are present if finding the optimal match (in the sense of \cref{sec:2:1}) is more difficult than usual, i.e., if several possible matches have similar loss function values. This can be made rigorous: given a user-defined tolerance $\delta>0$, we spot potential bifurcations by checking whether
\begin{equation}\label{eq:bif_tol}
    \textrm{loss}_{\textrm{second-best match}}<(1+\delta)\textrm{loss}_{\textrm{best match}}.
\end{equation}

If the above condition is true, i.e., if several matches perform similarly well (or badly), we decide not to ``trust'' the match and, rather than \emph{explicitly} interpolating the eigenvalue curves one by one, we proceed \emph{implicitly} as described in \cref{sec:3:0}. More specifically, we look at how the second-best match differs from the best one: any pair of indices $(\sigma_k,\tau_k)$ that is absent from the second-best match corresponds to one pair of eigenvalues (allegedly) undergoing bifurcation, namely, $\lambda_1^{\sigma_k}$ and $\lambda_2^{\tau_k}$. All such flagged eigenvalues would then be grouped together, and eventually dealt with by following the implicit treatment described in the previous section, cf.~\cref{fig:bif_qual} (right).

We have the following theoretical foundations of this idea.

\begin{figure}[t]
\centering
\includegraphics{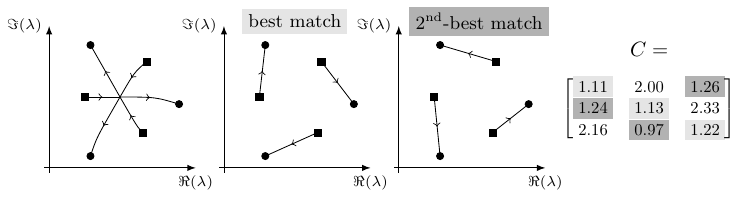}
\caption{(Continued from \cref{fig:bif_qual}.) Two different possible matches are shown in the middle and right plots. They correspond to the entries highlighted (in light gray and dark gray, respectively) in the cost matrix on the far right. Their losses are very similar, differing only by $\sim0.15\%$.}
\label{fig:bif_loc}
\end{figure}
    
\begin{proposition}
    Let $\delta>0$,  assume that a bifurcation is happening at $p=p^\star$, and set $p_1=p^\star-\varepsilon^-$ and $p_2=p^\star+\varepsilon^+$. For small enough $\varepsilon^-,\varepsilon^+>0$ (i.e., if the $p$-interval containing the bifurcation is small enough), the criterion \cref{eq:bif_tol} is true.
\end{proposition}

\begin{proof}[Sketch of proof]
    Under regularity assumptions on $\vL$, the eigenvalue curves approach the bifurcation point forming a ``star'' with uniformly spaced angles, see, e.g., \cite[Theorem 3]{Lancaster}. Specifically, the angles at which such curves enter the bifurcation point are interleaved with the angles at which they exit from it. See \cref{fig:bif_loc} (left) for a graphical representation.

    In the limit $\abs{p_2-p_1}\to 0$, this means that each eigenvalue $\lambda_1^i$ is equally far from its two closest eigenvalues $\lambda_2^j$, and vice versa. Now, consider two different matches: one performed clockwise, as in \cref{fig:bif_loc} (middle), and one performed counterclockwise, as in \cref{fig:bif_loc} (right). As $\abs{p_2-p_1}\to 0$, these matches have equal and optimal loss, so that $\textrm{loss}_{\textrm{second-best match}}-\textrm{loss}_{\textrm{best match}}\to 0$.
\end{proof}

We stress once more that the two (asymptotically) optimal matches in the proof above are equally ``bad'' from an approximation viewpoint, in the sense that they miss the bifurcation. We only use them to spot potential bifurcations, leaving the approximation task to the implicit method from \cref{sec:3:0}.

Unfortunately, most solvers for the optimal-matching problem compute only the optimal solution, and are incapable of computing suboptimal ones, as required by \cref{eq:bif_tol}. To circumvent this issue, we propose \cref{algo:bifurcation}. Essentially, we first find the optimal solution of the problem, in terms of the index tuples ${\bm\sigma}$ and ${\bm\tau}$, cf.~\cref{sec:2:1}. Then, for any $i$, we seek a solution of the problem that \emph{prohibits} matching the indices $\sigma_i$ and $\tau_i$. To this aim, we replace the entry $c_{\sigma_i\tau_i}$ of the cost matrix with $\infty$, and then execute the optimal-matching routine on this modified matrix. Then we compare the loss of the resulting suboptimal match with the optimal one, using \cref{eq:bif_tol}. If the criterion is satisfied, we conclude that a bifurcation may be present. Note that, as $i$ varies, the algorithm will necessarily explore the above-mentioned second-best match.

\begin{algorithm}[htb]
\caption{Flag eigenvalues that are involved in bifurcations}
\label{algo:bifurcation}
\begin{algorithmic}[1]
    \REQUIRE{cost matrix $C\in\bR^{N\times N}$ with entries as in \cref{eq:cost_entry}, tolerance $\delta>0$}
    \STATE{find optimal match: $({\bm\sigma},{\bm\tau})\gets\texttt{OptimalMatch}(C)$, cf.~\cref{sec:2:1}}
    \STATE{find optimal loss: $\texttt{loss}\gets\sum_{k=1}^Nc_{\sigma_k\tau_k}$}
    \FOR{$i=1,\ldots,N$}
        \STATE{define $C^\#$ as a copy of $C$, where entry $c_{\sigma_i\tau_i}$ has been replaced by $\infty$}
        \STATE{find match: $({\bm\sigma}^\#,{\bm\tau}^\#)\gets\texttt{OptimalMatch}(C^\#)$}
        \STATE{find optimal loss: $\texttt{loss}^\#\gets\sum_{k=1}^Nc_{\sigma_k^\#\tau_k^\#}$}
        \IF{$\texttt{loss}^\#<(1+\delta)\cdot\texttt{loss}$}
            \STATE{find differences between $({\bm\sigma},{\bm\tau})$ and $({\bm\sigma}^\#,{\bm\tau}^\#)$: any pair of indices that appears in $({\bm\sigma},{\bm\tau})$ but not in $({\bm\sigma}^\#,{\bm\tau}^\#)$ is involved in bifurcations}
        \ENDIF
    \ENDFOR
\end{algorithmic}
\end{algorithm}

For generality, it is important to make the following observation.

\begin{remark}\label{rem:unbalancedbifurcation}
    In the above discussion on bifurcations, we assumed that no eigenvalues migrate out through the contour, i.e., $\Delta N=0$. Our proposed strategy easily generalizes to the case $\Delta N\neq 0$. However, an intrinsic limitation is present: the index $i$ in \cref{algo:bifurcation} only loops over matched eigenvalues, disregarding any migrating ones. As such, no migrating eigenvalues can be flagged as part of a bifurcation.
\end{remark}

\section{Adaptive sampling}\label{sec:4}
So far, we have been working under the assumption that the grid of collocation points $\{p_j\}_{j=1}^S$ is given in advance. Still, the collocation points have a great effect on the approximation quality. For instance, in case of eigenvalue migration, it greatly helps to have a collocation point where the eigenvalue is just about to cross the contour. Moreover, the smoothness of the eigenvalue curves determines how many samples are needed to well approximate them. These observations behoove us to be extremely careful when choosing the collocation points.

However, all the above considerations involve features of the eigenvalues that are unknown \emph{a priori}. As such, choosing the collocation points in advance can lead to disappointing results. To solve this issue, we propose a fully non-intrusive strategy for \emph{a posteriori} adaptivity, allowing a fully automated selection of the collocation points' number and locations.

\begin{algorithm}[htb]
\caption{Adaptive sampling based on a predictor-corrector scheme}
\label{algo:adaptive}
\begin{algorithmic}[1]
    \REQUIRE{initial collocation points $p_j$ (e.g., a uniform grid over $\mathcal{P}$), tolerance $\varepsilon>0$}
    \STATE{$\texttt{converged}\gets\texttt{False}$}
    \WHILE{not \texttt{converged}}
        \STATE{build approximations $\tilde\lambda^i$ over the collocation points, as in \cref{sec:2,sec:3}}\label{item:item}
        \STATE{define test points as refinement of current collocation points: $\hat p_j:=\frac{p_j+p_{j+1}}2$}\label{item:refine}
        \STATE{$\texttt{converged}\gets\texttt{True}$}
        \FOR{all test points $\hat p_j$}
            \STATE{find $\{\lambda^i(\hat p_j)\}_i$ by solving non-parametric eigenvalue problem}
            \STATE{compare predicted eigenvalues $\{\tilde\lambda^i(\hat p_j)\}_i$ with exact eigenvalues $\{\lambda^i(\hat p_j)\}_i$}
            \IF{approximation accuracy is worse than $\varepsilon$}\label{item:test}
                \STATE{add $\hat p_j$ to the collocation points}
                \STATE{$\texttt{converged}\gets\texttt{False}$}
            \ENDIF
        \ENDFOR
    \ENDWHILE
\end{algorithmic}
\end{algorithm}

In \cref{algo:adaptive}, we summarize our approach, which is based on a \emph{predictor-corrector} framework that was already considered, e.g., in \cite{NP21}. In our pseudo-code, the only step that requires further discussion is line~\ref{item:test}: what quantity is compared to $\varepsilon$ in order to determine whether the test point $\hat p_j$ should become a collocation point? Our answer relies on the matching strategy introduced in \cref{sec:2:1}, as we proceed to discuss.

Let $\{\tilde\lambda^i(\hat p_j)\}_{i=1}^{\tilde N}$ be the approximate eigenvalues at $p=\hat p_j$. %In line with \cref{sec:2:2}, we assume that any migrating eigenvalue has been removed. %
In addition, let $\{\lambda^i(\hat p_j)\}_{i=1}^N$ be the exact eigenvalues at $p=\hat p_j$, computed by the non-parametric eigenvalue solver. In order to assess the approximation quality, we perform an optimization step as in \cref{sec:2:2}, computing the minimal-loss match of $\{\tilde\lambda^i(\hat p_j)\}_{i=1}^{\tilde N}$ with $\{\lambda^i(\hat p_j)\}_{i=1}^N$. We deem the current approximation accurate if all the costs involved in the optimal match (i.e., the distances between pairs of matched eigenvalues) are below the tolerance $\varepsilon$. Explicitly: given the optimal match indices $\bm\sigma$ and $\bm\tau$, cf.~\cref{sec:2:2}, we flag $\hat p_j$ for subsequent refinement if
\begin{equation*}
    \max_{i=1,\ldots,\min\{\tilde N, N\}}\abs{\tilde\lambda^{\sigma_i}(\hat p_j)-\lambda^{\tau_i}(\hat p_j)}>\varepsilon.
\end{equation*}
With this strategy, the algorithm is expected to take new samples where (i) wrong matches are performed, (ii) the eigenvalue curves are badly approximated, (iii) errors happen in bifurcation flagging, or (iv) poorly identified eigenvalue migrations happen.

Concerning this last point, we note that it is possible for the number of predicted and reference eigenvalues to be different: with the above notation, $\tilde N\neq N$. This case corresponds to eigenvalues that are predicted to cross the contour too ``early'' or too ``late'' (in terms of $p$). The strategy outlined above is very lenient about this effect: whenever the number of eigenvalues is predicted wrong, we disregard any extra/missing eigenvalues in the error computation. In our numerical experiments, we have observed that, in some cases, this choice can lead to too few refinements around values of $p$ at which migrations happen. This potentially unwanted behavior can be counteracted by deciding to treat the case with $\tilde N\neq N$ as an automatic fail of the adaptivity test. Of course, this idea runs the opposite risk of refining too aggressively around parameters where migrations happen.

\begin{remark}\label{rem:fooled}
    The proposed strategy for adaptivity is fully data-driven. Accordingly, it may be fooled whenever the data does not provide enough information. For instance, \cref{algo:adaptive} may reach convergence without identifying eigenvalue crossings or bifurcations, as long as such effects are sufficiently localized. See, e.g., \cref{fig:fooled}.

    \begin{figure}[ht]
    \centering
    \includegraphics{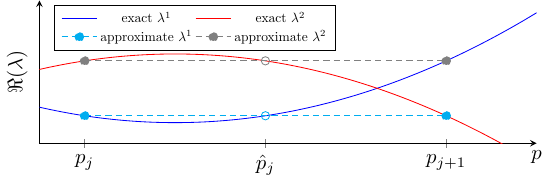}
    \caption{The testing error at $\hat p_j=\frac{p_j+p_{j+1}}2$ is zero. This prevents refinements near the mislabeled eigenvalue crossing.}
    \label{fig:fooled}
    \end{figure}

    Although this is a common limitation of data-driven adaptivity \cite{NP21}, partial countermeasures are possible. For instance, to mention a few, one could (i) incorporate eigenvector information in the match step, as in \cref{rem:vectors}, (ii) start with a finer initial grid of collocation points, or (iii) use a finer test set, e.g., by adding also $\frac{3p_j+p_{j+1}}4$ and $\frac{p_j+3p_{j+1}}4$ as test points in line \ref{item:refine} of \cref{algo:adaptive}. Despite improving our method's robustness, none of the above ideas guarantees success: counterexamples in the style of \cref{fig:fooled} may be designed to fool any data-driven strategy.
\end{remark}

\section{Numerical examples}\label{sec:5}
We now present some numerical evidence of the effectiveness of our proposed approach. We make our basic Python implementation of the method freely available at \url{https://github.com/pradovera/pEVP_match}, where we also include some code to reproduce our results below.

\subsection{Roots of cubic bivariate polynomial}\label{sec:5:1}
As a first artificial example, we consider the problem of identifying the roots (with respect to $\lambda$) of the polynomial
\begin{equation*}
    \pi(\lambda,p)=\lambda^3+(p-2)\lambda+(2p-1),
\end{equation*}
with parameter $p\in\mathcal{P}=[-50,50]$. To this aim, we cast the problem as a parametric linear eigenvalue problem involving a companion matrix:
\begin{equation}\label{eq:5:1:evp}
    \text{find }\lambda(p)\in\bC\text{ s.t. }\exists\vx\in\bC^3\setminus\{\vzero\}\ :\ \underbrace{\left(\begin{bmatrix}
        0&0&1-2p\\
        1&0&2-p\\
        0&1&0
    \end{bmatrix}-\lambda\vI_3\right)}_{\vL(\lambda,p)}\vx=\vzero,
\end{equation}
where $\vI_3$ is the identity matrix of size 3.

A cumbersome explicit formula is available\footnote{Given $\alpha(p)=\sqrt[3]{\frac{\sqrt{4 p^3 + 84 p^2 - 60 p - 5}}{6\sqrt{3}} - p + \frac12}$ and $\beta(p)=\frac{p - 2}{3\alpha(p)}$, the solutions of \cref{eq:5:1:evp} are \begin{equation*}
    \lambda^{1,2,3}(p) = \left\{\alpha(p)-\beta(p),-\frac{1-i\sqrt{3}}{2}\alpha(p)+\frac{1 + i\sqrt{3}}{2}\beta(p),-\frac{1 + i\sqrt{3}}{2}\alpha(p)+\frac{1 - i\sqrt{3}}{2}\beta(p)\right\}.
\end{equation*}} for the three parametric roots. Still, given the scope of the paper, we try to solve the problem with our proposed method, effectively ignoring the precise structure of the matrix-valued function $\vL$ in \cref{eq:5:1:evp}. Specifically, we choose to employ Beyn's method as the non-parametric eigenvalue solver in our algorithm, which only assumes that it is possible to solve linear systems involving $\vL(\lambda,p)$ at arbitrary values of $\lambda$ and $p$.

\begin{figure}[t]
\centering
\includegraphics{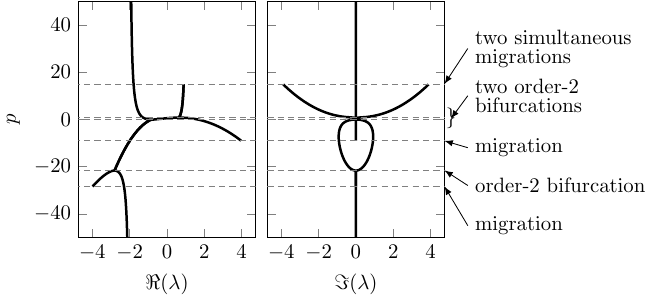}
\caption{Real (left) and imaginary (right) parts of the solutions of \cref{eq:5:1:evp}. The curves are plotted only over the region of interest, i.e., $\abs{\lambda}\leq4$.}
\label{fig:1:qual}
\end{figure}

As contour for Beyn's method, we choose a radius-$4$ circle centered at $0$. The behavior of the eigenvalues in this region is displayed in \cref{fig:1:qual}. On the right-hand side of the figure, we also indicate the major events that happen as the parameter $p$ varies within its range: four eigenvalue migrations (at $p\simeq-28.5$, $p\simeq-9.17$, and two at $p\simeq14.8$) and 3 order-2 bifurcations (at $p\simeq-21.7$, $p\simeq-0.075$, and $p\simeq0.76$). Note that the latter two bifurcations are quite close to each other, resembling, in fact, a single order-3 bifurcation. Of course, such features of the eigenvalues are unknown \emph{a priori}.

We apply our proposed method to approximate the eigenvalue curves $\lambda^i(p)$. More specifically, we employ the adaptive strategy described in \cref{sec:4} with tolerance $\varepsilon=10^{-2}$, starting with an initial grid containing only the endpoints $p_1=-50$ and $p_2=50$. We approximate migrations using the extrapolation-based strategy from \cref{sec:2:3}. Bifurcations are handled as described in \cref{sec:3}, with $\delta=0.1$. Piecewise-linear interpolation is used to build the approximate eigenvalue curves. To discretize the contour integrals required by Beyn's method, we use the trapezoidal rule \cref{eq:beyn_quad} with $N_z=25$ quadrature points. Also in Beyn's method, we set $K=1$ and $m=5$.

\begin{figure}[t]
\centering
\includegraphics{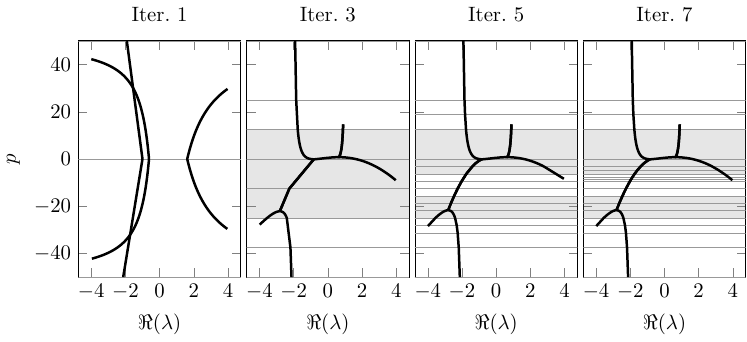}
\caption{Approximate eigenvalue curves (real part) at different stages of the adaptive algorithm (iteration indices are shown on top). Horizontal lines are placed at the collocation points' locations. Grayed-out areas denote regions where bifurcations are detected.}
\label{fig:1:iter}
\end{figure}

In \cref{fig:1:iter} we show the approximate eigenvalue curves at some iterations of the adaptive procedure, which terminates at the seventh iteration. Initially, the approximation quality is poor over most of the parameter range, resulting in uniform refinements of the sampling grid. At later iterations, the algorithm is able to automatically identify locations where the current approximation quality is good (enough), e.g., near the endpoints of the parameter range. Accordingly, no collocation points are added there. On the other hand, further refinements are carried out in critical regions, where the behavior of the eigenvalue curves is not yet understood well enough.

In the same plot, we also see how the ranges of $p$ where bifurcations occur seem to be identified well. Sometimes the grayed-out bands extend one interval up or down from the actual bifurcation. This means that an implicit representation of the eigenvalue curves (see \cref{sec:3:0}) is used despite no actual bifurcation happening. This is not an issue: just because the eigenvalue curves are represented implicitly, it does not mean that the algorithm is placing an (approximate) bifurcation there.

We note that an order-3 bifurcation (thus, involving all three eigenvalue curves) is placed around $p=0$. This yields an approximation of great quality. In fact, given the structure of the problem, the interpolatory order-3 implicit representation of the eigenvalues is exact. Note, however, that our data-driven algorithm is completely unaware of this.

On the other hand, refinements are performed in the region just below $p=0$. As is apparent by looking at the rightmost plot in \cref{fig:1:iter}, this process is driven by the eigenvalue migration at $p\simeq-9.17$ (where the collocation points are most dense), and not by the bifurcation.

\begin{figure}[t]
\centering
\includegraphics{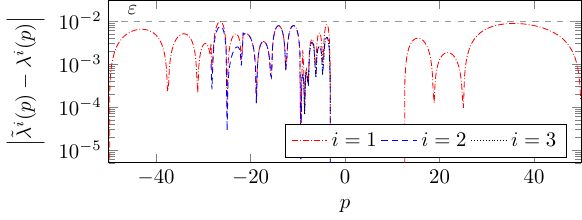}
\caption{Approximation error in the eigenvalue curves. The horizontal line denotes the tolerance $\varepsilon$.}
\label{fig:1:error}
\end{figure}

The approximation error shown in \cref{fig:1:error}, comparing the approximate eigenvalue curves with the exact ones, quantitatively confirms our conclusions above. In particular, the order-3 implicit representation of the eigenvalue curves near $p=0$ yields errors of the order of round-off. More generally, we see that the error is uniformly below the tolerance $\varepsilon$. Note that, since our adaptive sampling strategy is data-driven, this might not be guaranteed in general: sub-grid effects (e.g., unidentified crossings, cf.~\cref{rem:fooled}) may increase the error above the tolerance in between collocation points.

\subsection{Waveguide modes under geometric variations}\label{sec:5:3}

We consider a time-harmonic wave propagation problem in a waveguide whose shape is subject to parametric variations. The study of the parametric waveguide eigenvalues and eigenmodes is extremely useful in modeling. For instance, one sometimes aims at optimizing the waveguide's shape to achieve some target filtering capabilities, e.g., for its use as a filter for electromagnetic waves or as a sound muffler.

Specifically, given a parameter $p\in\mathcal{P}=[0.2,0.8]$, we consider the simple 2D domain $\Omega=\Omega(p)\subset\bR^2$, defined as
\begin{align*}
    \Omega(p)=&\left(]0,3[\ \times\ ]0,1[\right)\setminus B_p(1,0)\\
    =&\left\{(x,y):0<x<3,0<y<1,(x-1)^2+y^2>p\right\},
\end{align*}
see \cref{fig:3:modes}, whose boundary is split into $\Gamma_{\textrm{out}}:=\{3\}\times]0,1[$ and $\Gamma_0=\Gamma_0(p):=\partial\Gamma(p)\setminus\Gamma_{\textrm{out}}$. We seek solutions $(\lambda,u)$, $u\not\equiv 0$, of the quadratic eigenvalue problem
\begin{equation}\label{eq:scatter}
\begin{cases}
    -\Delta u(x,y)=\lambda^2u(x,y)\quad&\text{for }(x,y)\in\Omega(p),\\
    \partial_\nu u(x,y)=0\quad&\text{for }(x,y)\in\Gamma_0(p),\\
    \partial_\nu u(x,y)=\mathrm{i}\lambda u(x,y)\quad&\text{for }(x,y)\in\Gamma_{\textrm{out}}.
\end{cases}
\end{equation}
Above, $\Delta$ denotes the 2D Laplace operator, $\nu$ is the outer normal vector to $\partial\Omega(p)$, and $\mathrm{i}$ is the imaginary unit. In computational practice, a discretization of the eigenvalue problem is considered instead of the continuous model in  \cref{eq:scatter}. Here, we choose to employ P2-finite elements on a triangular mesh.

\begin{figure}[t]
\centering
\includegraphics{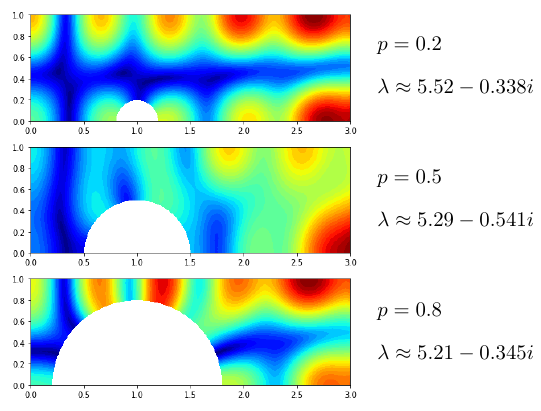}
\caption{From top to bottom, sample eigenmodes corresponding to different parameter values (reported on the right-hand side). The three eigenvalues have been selected from the same eigenvalue curve.}
\label{fig:3:modes}
\end{figure}

Note that the last equation, a $\lambda$-dependent Robin-type boundary condition on $\Gamma_{\textrm{out}}$, is the widely used first-order approximation of the Sommerfeld ``outgoing wave'' condition. It makes the eigenvalue problem quadratic. Moreover, any eigenvalue other than the trivial $\lambda=0$ must have negative imaginary part \cite{spectrum}.

Variations in the parameter $p$ change the domain of definition of the eigenvector $u$, as well as the domain where \cref{eq:scatter} is to be solved. Specifically, each value of $p$ entails an independent assembly of the finite-element matrices involved in the discretization of \cref{eq:scatter}, on top of meshing the varying domain $\Omega(p)$. In particular, the size of the discrete problem \cref{eq:pevp} may change with $p$. (In our implementation, as $p$ varies in $\mathcal{P}$, finite-element spaces have sizes varying between $6\cdot 10^3$ and $10^4$.)

This prevents most intrusive methods from being applied to the present parametric eigenvalue problem, at least in its current form. More in detail, to make the problem tractable via intrusive approaches, one first must cast the PDE model \cref{eq:scatter} over a parameter-\emph{independent} domain $\Omega^\star$ by means of a parameter-\emph{dependent} mapping $\Phi(p):\Omega^\star\to\Omega(p)$ \cite{shape2,shape1}. However, the construction of such mapping is often extremely laborious and very intrusive, requiring considerable user effort to recast the problem in a usable but, arguably, less natural form. Instead, the non-intrusiveness of our approach allows its use even though the eigenvalue samples at the collocation points are obtained using different meshes\footnote{Admittedly, the approximation of the eigen\emph{vector} curves remains an issue even in our non-intrusive framework, since one has to deal with eigenvectors of incompatible sizes. However, in our present experiment, only the eigen\emph{value} curves are of interest.}.

To approximate the eigenvalues of the above problem, we apply the same framework as in the previous experiment, with two main differences: (i) we choose a radius-$2$ disk centered at $5$ as contour in Beyn's method, and (ii) we use cubic splines for $p$-interpolation. Moreover, to account for the larger expected number of eigenvalues, we set a higher number of quadrature points $N_z=200$, as well as $K=5$ and $m=10$, in Beyn's method. Our adaptive algorithm terminates after selecting $24$ parameter values as collocation points.

\begin{figure}[t]
\centering
\includegraphics{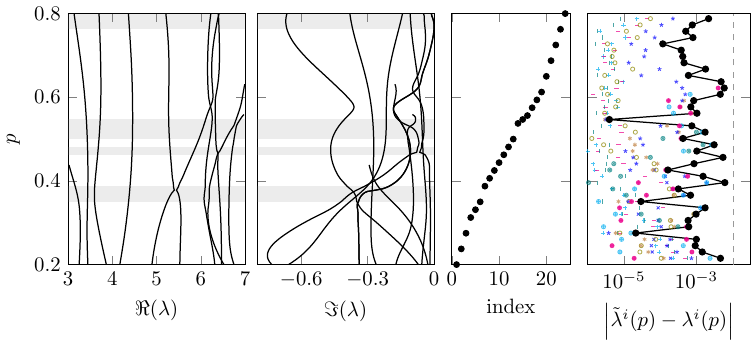}
\caption{From left to right: (1) real and (2) imaginary parts of the approximate eigenvalue curves (grayed-out areas denote regions where bifurcations are detected); (3) adaptively selected collocation points, sorted in increasing order, displayed by index; (4) approximation error in the eigenvalue curves over a uniform grid of 50 parameter values, with different symbols denoting different indices $i$ (the black line is the envelope of the maximum errors).}
\label{fig:3:app}
\end{figure}

The resulting approximate eigenvalue curves are shown in \cref{fig:3:app} (left). See also \cref{fig:qual} for a depiction of the approximate eigenvalue trajectories in the complex plane. Some eigenvalue migrations are present, since the number of eigenvalues identified by Beyn's method varies between $N=7$ and $N=10$. Some bifurcations also appear to be present. The algorithm seems to have no issues in their identification and approximation. We can also observe that some of the eigenvalue curves get very close to the real axis for some values of $p$. This is of practical interest, being a symptom of almost-resonating behavior of the waveguide.

The approximation error is shown in \cref{fig:3:app} (right). Note that, in this example, we do not have an explicit formula for the exact eigenvalues. As such, we compute the error by looking at the distance between the approximate eigenvalues and the eigenvalues obtained by expensively running Beyn's method at each of the chosen test values of $p$. We can observe that the tolerance is attained at all test points.

\subsection{Delayed control of a heat equation}\label{sec:5:2}

We consider a simple heat equation subject to delayed feedback, inspired by \cite{Caliskan09}: we seek a space-time solution $u=u(x,t)$ of
\begin{equation}\label{eq:heat}
\begin{cases}
    \partial_tu(x,t)=\kappa\partial_{xx}u(x,t)-f(u,x,t)\quad&\text{for }(x,t)\in(0,\pi)\times(0,\infty),\\
    u(0,t)=u(\pi,t)=0\quad&\text{for }t\in(0,\infty),\\
    u(x,0)=u_0(x)\quad&\text{for }x\in(0,\pi).\\
\end{cases}
\end{equation}
Above, $\kappa=0.02$, $u_0$ is some initial condition (unimportant for our discussion), and
\begin{equation*}
    f(u,x,t)=0.1u(x,t)+0.05u(x,t-1)+pu(x,t-2),\quad p\in\bR,
\end{equation*}
is a distributed feedback, with an instantaneous term and two delayed terms, with delays $\tau_1=1$ and $\tau_2=2$. Note the presence of the parameter $p$, corresponding to the strength of the second delayed feedback.

Interested in the stability properties of \cref{eq:heat}, we consider its frequency-domain formulation, obtained via the \emph{ansatz} $u(x,t)=v(x)e^{\lambda t}$, $\lambda\in\bC$. This leads to a nonlinear infinite-dimensional eigenvalue problem involving the Laplace operator $\partial_{xx}$, which, after spatial discretization (e.g., by finite differences), results in a problem like \cref{eq:pevp}. Specifically, if a uniform grid of $(M+1)$ points is used to discretize the interval $[0,\pi]$, centered finite differences lead to a size-$(M-1)$ problem that reads
\begin{equation*}
    \underbrace{\left(\kappa\left(\frac{M}\pi\right)^2\begin{bmatrix}
        2 & -1 & & & \\
        -1 & 2 & -1 & & \\
         & \ddots & \ddots & \ddots & \\
         & & -1 & 2 & -1\\
         & & & -1 & 2\\
    \end{bmatrix}+(\lambda+0.1+0.05e^{-\lambda}+pe^{-2\lambda})\vI_{M-1}\right)}_{\vL(\lambda,p)}\vx=0,
\end{equation*}
with $\vI_{M-1}$ being the size-($M-1$) identity matrix. In our tests, we set $M=5\cdot10^3$.

In order to approximate the eigenvalues of the above problem, we apply the same framework as in the previous experiment, with three main differences: (i) the parameter range is now $\mathcal{P}=[-0.1,0.1]$, (ii) since we expect stable eigenvalues (with negative real part), we choose a radius-$1$ disk centered at $-1$ as contour in Beyn's method, and (iii) still in Beyn's method, we set a higher number of quadrature points $N_z=10^3$, as well as $m=30$. Our adaptive algorithm terminates after selecting $60$ parameter values as collocation points.

\begin{figure}[t]
\centering
\includegraphics{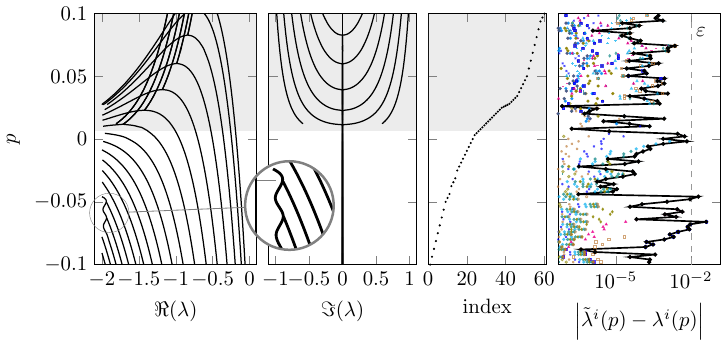}
\caption{From left to right: (1) real and (2) imaginary parts of the approximate eigenvalue curves; (3) adaptively selected collocation points, sorted in increasing order, displayed by index; (4) approximation error in the eigenvalue curves over a uniform grid of 100 parameter values, with different symbols denoting different indices $i$ (the black line is the envelope of the maximum errors).}
\label{fig:2:app}
\end{figure}

The resulting approximate eigenvalue curves are shown in \cref{fig:2:app}. We can observe many eigenvalue migrations at $\lambda=-2$, the leftmost point of the contour. Indeed, the number of eigenvalues identified by Beyn's method varies between $N=7$ (for $p\approx0.005$) and $N=18$ (for $p\approx-0.1$).

Some of the migrations around $p\approx-0.05$ were particularly difficult for the algorithm to deal with. Specifically, the migrating eigenvalue ($\approx-2$) was mistakenly matched to a not-yet-migrating eigenvalue ($>-2$) in two separate instances. As such, an oscillating trajectory of the eigenvalue curve is predicted, linking the migrating eigenvalue with the not-yet-migrating one. A magnified view of this is included in \cref{fig:2:app}.

Five bifurcations, with eigenvalues turning from real to complex conjugates, are also present. The algorithm seems to have no issues in their identification and approximation. Indeed, the collocation points are quite sparse for $p>0.028$, where three of the bifurcations happen. Notably, the collocation points in that region appear as sparse as in the region where no bifurcations happen, namely, $p<0$. This is evidence of the fact that, thanks to the strategy described in \cref{sec:3}, bifurcations do not degrade the approximation quality. This would not have been the case if all eigenvalues had been treated as semi-simple.

The approximation error is shown in \cref{fig:2:app} (right). As in the previous example, we do not have an explicit formula for the exact eigenvalues, and we use the values obtained by expensively running Beyn's method as reference. We can observe that the tolerance is not attained near the two badly identified eigenvalue migrations. Since they arise as ``sub-grid effects'', such higher-than-desired errors are invisible to our adaptive sampling strategy, cf.~\cref{rem:fooled}. As we show next, reducing the tolerance fixes this issue.

\begin{figure}[t]
\centering
\includegraphics{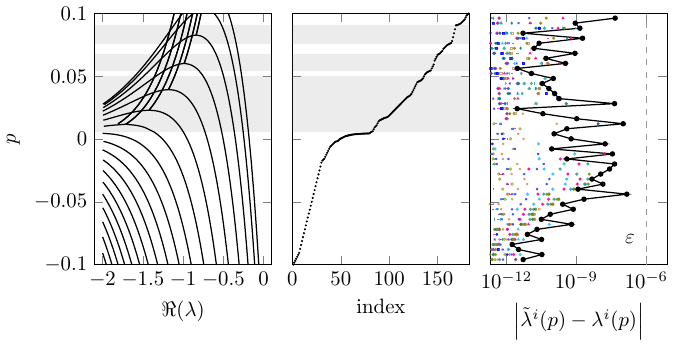}
\caption{Numerical results with reduced tolerance $\varepsilon=10^{-6}$. From left to right: (1) real part of the approximate eigenvalue curves; (2) adaptively selected collocation points, sorted in increasing order, displayed by index; (3) approximation error in the eigenvalue curves over a uniform grid of 100 parameter values, with different symbols denoting different indices $i$ (the black line is the envelope of the maximum errors).}
\label{fig:2:err_6}
\end{figure}

Now we repeat our simulation after reducing the tolerance to the much lower value $\varepsilon=10^{-6}$. To help the algorithm reach this lower tolerance, we also enlarge the stencil of the interpolation strategy, by using order-7 B-splines. The algorithm selects $182$ collocation points, displayed in \cref{fig:2:err_6} (middle). We observe a high density of collocation points around $p=0$. This seems to be functional to approximating the eigenvalue migration happening there, which was particularly hard to pinpoint for our method. Upon closer inspection of our simulation results, we were able to find the reason for this: the quadrature formula \cref{eq:beyn_quad} can lose accuracy if eigenvalues are too close to the integration contour (see \cite[Corollary 4.8]{Be12}). Accordingly, Beyn's method can sometimes be affected by numerical noise when computing eigenvalues that are about to migrate. In our simulation, this ended up leading to inaccuracies in the construction of one specific (migrating) eigenvalue curve by extrapolation, cf.~\cref{as:exactness} and \cref{rem:migrations}.

For validation, we show the approximation error in \cref{fig:2:err_6} (right). We see that the error landscape is always below the tolerance $\varepsilon$. This positive result shows that our method has the potential of reaching an accuracy much lower than the pragmatic standard $10^{-3}\div 10^{-2}$, even in rather complex applications.

As a final test to investigate how the number of collocation $p$-points varies as the tolerance is reduced, we repeat the above experiment for a collection of tolerance values between $10^{-6}$ and $10^{-1}$. We test two different $p$-interpolation strategies: low-order piecewise-linear interpolation and high-order degree-7 splines. From the results in \cref{fig:2:scaled}, we can observe that, as expected, the number of collocation points increases as the accuracy requirements become stricter. In particular, after a short preasymptotic regime where the two interpolation approaches seem comparable, the sample size increases fairly steeply for piecewise-linear interpolation, but only quite mildly when using splines. It is interesting to note that, in both cases, the number of collocation points roughly scales as one would expect when using the respective interpolation methods to approximate \emph{smooth} functions: the approximation error converges with order $2$ for piecewise-linear interpolation, and with order $8$ for degree-7 splines. This is an extremely positive result, considering that our approximation targets (the eigenvalue curves), due to migrations and bifurcations, are not always smooth.

\begin{figure}[t]
\centering
\includegraphics{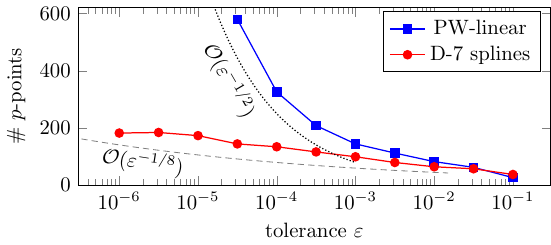}
\caption{Number of collocation $p$-points selected by the adaptive algorithm for various tolerance values. The dotted and dashed lines are the reference scalings $\mathcal{O}(\varepsilon^{-1/2})$ and $\mathcal{O}(\varepsilon^{-1/8})$, which piecewise-linear functions and degree-7 splines achieve in approximating smooth functions, respectively.}
\label{fig:2:scaled}
\end{figure}

Lowering the tolerance further than $10^{-6}$ proved fruitless in our experiments, resulting in a non-converging adaptivity loop. This is related to \cref{as:exactness}, which requires the underlying non-parametric eigenvalue solver to be accurate enough. Essentially, if $\varepsilon$ is too small, the lower accuracy of Beyn's method disrupts the convergence of our approach. A good convergence behavior may be recovered even for small $\varepsilon$, by increasing the number of quadrature points in Beyn's method. Ultimately, we expect our method to be robust enough to reach substantial accuracy (say, $10^{-10}$), as long as the noise in the data is kept under sufficient control. We show this with our last numerical test.

\subsection{Modes of a parametric radio-frequency gun cavity}\label{sec:5:4}

We consider the so-called ``gun'' eigenvalue problem from the \emph{NLEVP} collection \cite{NLEVP}, which models the propagation of electromagnetic waves in a complex 3D cavity. To introduce a parametric dependence, we allow the index of refraction $p$ of the medium filling the cavity (relative to vacuum) to vary in $\mathcal{P}=[1,1.4]$. The resulting problem reads
\begin{equation}\label{eq:gun}
    \vL(\lambda,p)=\vA_0+p^2\lambda\vA_1+\mathrm{i}p\left(\sqrt{\lambda}\vA_2+\sqrt{\lambda-\alpha^2}\vA_3\right),
\end{equation}
where $\alpha=108.8774$ is a constant and $\vA_j\in\bC^{n\times n}$, $j=1,\ldots,4$, are sparse matrices of size $n=9956$ resulting from a 3D FEM-based discretization of Maxwell's equations in the cavity. From a physical standpoint, this formulation is a generalization of \cref{eq:scatter}, with $\lambda$ replacing $\lambda^2$, and with more complex boundary conditions, giving rise to further nonlinearities in $\lambda$. Thanks to its non-intrusiveness, our algorithm can be applied even in this setting without any modifications, despite the more complex $\lambda$-dependence.

For physical reasons, all eigenvalues have positive imaginary part, and the most ``interesting'' eigenvalues are located just to the right of the branch point at $\alpha^2$. As such, we seek eigenvalues in a ball of radius $\rho=1.5\cdot 10^4$ centered at $\overline{\lambda}=3\cdot 10^4$. To show the effectiveness of our method, we set a tolerance of $10^{-6}$ for the adaptive sampling. Note that, due to the large magnitude of the eigenvalues, this corresponds to a very small $\sim3\cdot10^{-11}$ \emph{relative} tolerance. Otherwise, the numerical setup is the same as in the previous section.

\begin{figure}[t]
\centering
\includegraphics{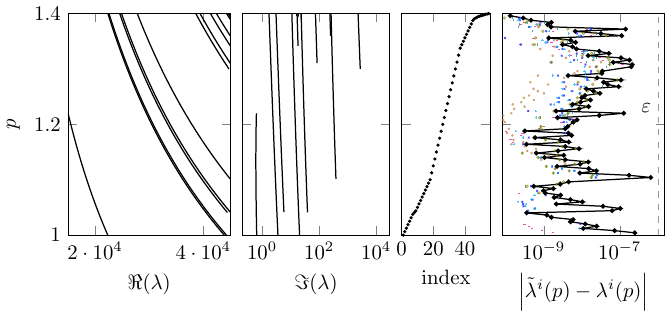}
\caption{From left to right: (1) real and (2) imaginary parts of the approximate eigenvalue curves; (3) adaptively selected collocation points, sorted in increasing order, displayed by index; (4) approximation error in the eigenvalue curves over a uniform grid of 100 parameter values, with different symbols denoting different indices $i$ (the black line is the envelope of the maximum errors).}
\label{fig:4:error}
\end{figure}

The algorithm ends after having selected 55 parameter values as collocation points. In \cref{fig:4:error} (left), we can see how, in this test, except for migrations, the $13$ eigenvalue curves that are present stay smooth and do not display bifurcations. This ultimately helps in keeping the number of collocation points relatively low. The algorithm places a high concentration of collocation points only near $p=1.4$. This is due to two extremely short eigenvalue curves, whose domains are approximately $p\in[1.395, 1.4]$, and which are barely visible near the top-right corner of \cref{fig:4:error} (left). As in the previous experiments, the validation error is uniformly below the tolerance, cf.~\cref{fig:4:error} (right). This is particularly impressive in the present test, due to the extremely low tolerance.

\begin{table}[t]
\centering
\begin{tabular}{c|c|c|c|}
     & \multicolumn{2}{|c|}{Training phase} & Evaluation \\
\cline{2-3}
     & Sampling of & Other steps (match, & at single new \\
     & hi-fi model & flag, adapt, ...) & value of $p$\\
\hline
    Hi-fi model & - & - & $6.04\cdot10^2$ \\
\hline
    Approximation & $6.54\cdot10^4$ & $2.90\cdot10^{-1}$ & $9.15\cdot10^{-3}$ \\
\hline
    \multicolumn{2}{c}{} & Speedup: & $6.60\cdot10^4$ \\
\end{tabular}
\caption{Timing (in seconds) of training and deployment for the ``gun'' problem. The code is available at \texttt{https://github.com/pradovera/pEVP\_match}. The evaluation times (rightmost column) are averaged over 100 sample $p$-points.}
\label{tab:4}
\end{table}

In \cref{tab:4} we look at the timing of the algorithm. In total, it takes about $18$ hours to build the approximation on a standard workstation with a 6-core 3.5-GHz Intel\textsuperscript{\textregistered} processor. By far, most of the time is spent using Beyn's method to solve the non-parametric nonlinear eigenvalue problem, namely, solving linear systems involving $\vL$. When it comes to predicting the eigenvalue curves at a new value of $p$, our approximate model achieves a remarkable speedup of more than four orders of magnitude.

\section{Conclusions}\label{sec:conclusions}

We have presented a novel general-purpose algorithm for solving parametric (nonlinear) eigenvalue problems. Our method relies on high-accuracy solutions of \emph{non-parametric} versions of the problem, and then combines the obtained information over the parameter range. To extend the scope of potential applications of our method, non-parametric eigenvalue solvers based on contour integration are considered. These only require non-intrusive access to the parametric matrix defining the target problem via matrix-vector products. We have showcased the effectiveness of our method with a suite of numerical tests, even involving rather complex parametric dependencies, which would have hindered (or prevented the use of) intrusive approaches.

Many aspects of the modeling-over-$p$ phase require special care. In our discussion, we have focused on the issues related to the prominent and common effects of \emph{eigenvalue migrations} and \emph{bifurcations}. Effective ways of dealing with these issues were presented, allowing our algorithm to achieve remarkably high approximation accuracy. Our proposed adaptive sampling strategy is pivotal to achieve such accuracies, without making any \emph{a priori} assumptions on the behavior of the eigenvalues.

In our presentation, we have focused on the approximation of the parametric eigenvalues, essentially disregarding eigenvectors. Nevertheless, much of what we have discussed can be employed to obtain an approximation of \emph{eigenvector curves}. Notably, eigenvector information can be exploited when performing the ``match'' step, by including information on the eigenvector distance in the cost matrix (see, e.g., \cite{NP21,bertrand_datadriven_2023}). In the effort of generalizing the concepts introduced in this paper to eigenvectors, the main difficulty is related to bifurcations: eigenvectors may lose meaning at bifurcation points, making it necessary to look at eigen\emph{spaces} instead \cite{cont3}. Moreover, the possibility of defective eigenvalues greatly complicates the eigenspace approximation endeavor (and its implementation). We leave this as a direction for further research.

Lastly, we note that our discussion focused on the single-parameter case. This setting is rich enough to justify our work, which can serve as a building block to reach the higher-dimensional case $\mathbf{p}\in\bR^{n_p}$, $n_p>1$. This often arises in practice, especially in uncertainty quantification, where, in fact, one of the most notable applications is ``$n_p=\infty$'' \cite{Andreev2012}. In our opinion, the core of our work (match step, interpolation, adaptivity) may be generalized to more than one parameter. Notably, whenever $n_p>1$, \emph{locally adaptive sparse grids} \cite{NP21} naturally generalize our adaptive sampling and interpolation strategies, without fully incurring the \emph{curse of dimension} as naive tensor sampling would. On the other hand, some of the more technical aspects of our method (especially in the approximation of bifurcations) are not so easy to extend to more parameters. Many interesting open questions remain in this direction, in terms of both modeling and implementation.

\bibliographystyle{siamplain}
\bibliography{references}
\end{document}

%% file: ex_shared.tex
\usepackage{cite}
\usepackage{amsfonts}
\usepackage{graphicx}
\usepackage{epstopdf}
\usepackage{algorithmic}
\usepackage{bm}
\ifpdf
  \DeclareGraphicsExtensions{.eps,.pdf,.png,.jpg}
\else
  \DeclareGraphicsExtensions{.eps}
\fi

\newsiamremark{remark}{Remark}
\newsiamremark{hypothesis}{Hypothesis}
\crefname{hypothesis}{Hypothesis}{Hypotheses}
\newsiamthm{claim}{Claim}
\crefname{enumi}{Step}{Steps}

\headers{Solving general parametric eigenproblems}{D. Pradovera and A. Borghi}

\title{Match-based solution of general parametric eigenvalue problems\thanks{\today~version.\funding{Alessandro Borghi acknowledges funding by the Deutsche Forschungsgemeinschaft (DFG, German Research Foundation) - 384950143 as part of GRK2433 DAEDALUS.}
}}

\author{Davide Pradovera\thanks{Department of Mathematics, University of Vienna, Vienna, Austria \& Department of Mathematics, KTH Royal Institute of Technology, Stockholm, Sweden
  (\email{davidepr@kth.se}, \url{https://pradovera.github.io}).}
\and Alessandro Borghi\thanks{Institute of Mathematics, Technical University Berlin, Berlin, Germany
  (\email{borghi@tu-berlin.de}, \url{https://www.tu.berlin/fgmso/alessandro-borghi}).}}

\usepackage{amsopn}

\newcommand*{\abs}[1]{\left|#1\right|}

\newcommand{\vx}{\mathbf{x}}

\newcommand{\vA}{\mathbf{A}}
\newcommand{\vB}{\mathbf{B}}
\newcommand{\vF}{\mathbf{F}}
\newcommand{\vG}{\mathbf{G}}
\newcommand{\vH}{\mathbf{H}}
\newcommand{\vI}{\mathbf{I}}
\newcommand{\vJ}{\mathbf{J}}
\newcommand{\vL}{\mathbf{L}}

\newcommand{\vR}{\mathbf{R}}
\newcommand{\vSigma}{\bm\Sigma}
\newcommand{\vU}{\mathbf{U}}
\newcommand{\vV}{\mathbf{V}}
\newcommand{\vX}{\mathbf{X}}
\newcommand{\vY}{\mathbf{Y}}
\newcommand{\vzero}{{\bm 0}}
\newcommand{\bC}{\mathbb{C}}
\newcommand{\bN}{\mathbb{N}}
\newcommand{\bR}{\mathbb{R}}

\makeatletter
\newcommand*{\addFileDependency}[1]{% argument=file name and extension
  \typeout{(#1)}% latexmk will find this if $recorder=0 (however, in that case, it will ignore #1 if it is a .aux or .pdf file etc and it exists! if it doesn't exist, it will appear in the list of dependents regardless)
  \@addtofilelist{#1}% if you want it to appear in \listfiles, not really necessary and latexmk doesn't use this
  \IfFileExists{#1}{}{\typeout{No file #1.}}% latexmk will find this message if #1 doesn't exist (yet)
}
\makeatother